\documentclass[12pt]{amsart} 

\usepackage{amssymb,amsmath,amsthm}
\usepackage{simplemargins, verbatim}
\usepackage{pinlabel}
\usepackage{float}
\usepackage{graphicx}
\usepackage{enumerate}
\usepackage{hyperref}
\hypersetup{
    colorlinks,%
    citecolor=black,%
    filecolor=black,%
    linkcolor=black,%
    urlcolor=black
}

\setallmargins{1.5 in}

\renewcommand{\epsilon}{\varepsilon}
\renewcommand{\setminus}{\smallsetminus}

\theoremstyle{plain}
\newtheorem{corollary}{Corollary}[section]
\newtheorem{lemma}[corollary]{Lemma}
\newtheorem{proposition}[corollary]{Proposition}

\newtheorem{theorem}[corollary]{Theorem}

\theoremstyle{definition}

\theoremstyle{remark}

\newcommand{\Prob}{{\mathbb P}}

\newcommand{\N}{{\mathbb N}} 
\newcommand{\E}{{\mathbb E}}
\newcommand{\R}{{\mathbb R}}
\newcommand{\C}{{\mathbb C}}

\newcommand{\Half}{{\mathbb H}}
\newcommand{\D}{{\mathbb D}}
\newcommand{\1}{{\rm 1\hspace*{-0.4ex}\rule{0.1ex}{1.52ex}\hspace*{0.2ex}}}

\newcommand{\F}{{\mathcal F}}

\newcommand{\exc}{{\mathcal E}}

\newcommand{\bd}{\partial}

\renewcommand{\Im}{{\rm Im}}

\newcommand{\SLE}{{\mathrm{SLE}}}

\newcommand{\mwhere}{{\ \ \text{where} \ \ }}

\newcommand{\rd}{{d}}
\newcommand{\dd}{{\ \rd}}

\newcommand{\adz}{{\ |\rd z|}}

\newcommand{\dt}{{\dd t}}

\newcommand{\dxy}{{\dd x \dd y}}

\newcommand{\ceq}{\;{\mathrel{\mathop:}=}\;}
\newcommand{\limints}{\idotsint\limits}
\newcommand{\Sig}{{\mathcal S}}

\title[Regularity of SLE, crossings and rough paths]{Regularity of Schramm-Loewner Evolutions, annular crossings, and rough path theory}
\author{Brent M. Werness}
\address{Department of Mathematics\\ University of Chicago\\ 5734 University Avenue\\ Chicago, IL 60637-1546}
\email{bwerness@math.uchicago.edu}

\begin{document}
\begin{abstract}
	When studying stochastic processes, it is often fruitful to understand several different notions of regularity.  One such notion is the optimal H\"older exponent obtainable under reparametrization.  In this paper, we show that chordal $\SLE_\kappa$ in the unit disk for $\kappa \le 4$ can be reparametrized to be H\"older continuous of any order up to $1/(1+\kappa/8)$. 
	 
	From this, we obtain that the Young integral is well defined along such $SLE_\kappa$ paths with probability one, and hence that $\SLE_\kappa$ admits a path-wise notion of integration.  This allows us to consider the expected signature of $\SLE$, as defined in rough path theory, and to give a precise formula for its first three gradings.
	
	The main technical result required is a uniform bound on the probability that an $\SLE_\kappa$ crosses an annulus $k$-distinct times.
\end{abstract}
\maketitle
\section{Introduction}\label{intro}

Oded Schramm introduced Schramm-Loewener Evolutions ($\SLE$) as a stochastic process to serve as the scaling limit of various discrete models from statistical physics believed to be conformally invariant in the limit \cite{first}.  It has successfully been used to study a number of such processes (for example, loop-erased random walk and uniform spanning-tree \cite{loop}, percolation exploration process \cite{perc}, Gaussian free field interfaces \cite{gff}, and Ising model cluster boundaries \cite{ising1}).  

To define $\SLE$, it is convenient to parametrize the curve so that the half-plane capacity increases linearly and deterministically with time -- a change which allowed the use of a form of the Loewner differential equation.  This has proven an extremely fruitful point of view, enabling the definition of $SLE_\kappa$ and the proof of all of the above convergence results.

However, in doing so, the original parametrizations of the discrete models are lost along with any information about the regularity of these parameterizations.  To try and recover some information on the possible regularity properties of these parametrizations, it is reasonable to ask the question: what are the best regularity properties that $\SLE_\kappa$ curves can have under any \emph{arbitrary} reparametrization?

Regularity of $\SLE_\kappa$ under the capacity parametrization is well understood. In \cite{holder2}, Johansson Viklund and Lawler prove a conjecture of Lind from \cite{holder1} that for chordal $\SLE_\kappa$ parametrized by capacity, the optimal H\"older exponent is
\[
\alpha_0 = \min\Big\{\frac{1}{2}, 1-\frac{\kappa}{24+2\kappa-8\sqrt{8+\kappa}}\Big\}.
\]

However, this value differs greatly from what one might expect.  In \cite{beffara}, Beffara shows that the almost sure Hausdorff dimension of a chordal $\SLE_\kappa$ is $1+\kappa/8$. A $d$-dimensional curve $\gamma$ cannot be reparametrized to be H\"older continuous of any order greater than $1/d$, and intuition from other stochastic processes implies that $\SLE_\kappa$ should be able to be reparametrized to be H\"older continuous of all remaining orders, which is not what we see under capacity parametrization.

In this paper, we answer this question for $\kappa \le 4$ and show that the best possible result is true.
\begin{theorem}\label{regularityTheorem}
	Fix $0 \le \kappa \le 4$ and let $\gamma:[0,\infty] \rightarrow \D$ be a chordal $\SLE_\kappa$ from $1$ to $-1$ in $\D$ and $d = 1 + \kappa/8$ be its almost sure Hausorff dimension.  Then, with probability one the following holds: 
	\begin{itemize}
		\item for any $\alpha < 1/d$, $\gamma$ can be reparametrized as a curve $\tilde\gamma :[0,1] \rightarrow D$ which is be H\"older continuous of order $\alpha$, and 
		\item for any $\alpha > 1/d$, $\gamma$ cannot be reparametrized as a curve $\tilde\gamma :[0,1] \rightarrow D$ which is be H\"older continuous of order $\alpha$.
	\end{itemize}
\end{theorem}

The critical case of $\alpha = 1/d$ is still open, however it is natural to conjecture that it cannot be reparametrized to be H\"older continuous of this order.  

With this result, we are able to provide a few preliminary results in the rough path theory of $\SLE$.  

First, we obtain a definition integration against a $\SLE_\kappa$ curve.  In particular, this result shows that $\SLE_\kappa$ for $\kappa \le 4$ has finite $d$-variation for some $d < 2$ in the sense used in \cite{rough1} and thus both the Young integral and the integral of Lions as defined in \cite{rough2} give a way of almost surely integrating path-wise along an $\SLE$ curve, and moreover iterating such integrals to provide almost sure existence of differential equations driven by an $\SLE$ curve. 

Second, we provide a partial computation of the expected signature for $\SLE_\kappa$ in the disk.  In rough path theory, the expected signature is a non-commutative power series which is regarded as a kind of non-commutative Laplace transform for paths. It is believed to characterize the measure on the path up to an appropriate sense of equivalence of paths \cite{expectedRough,uniqueSig}.  We provide a computation of the first three gradings of this non-commutative power series.

The main technical tool used to prove these results is a result by Aizenman and Burchard in \cite{AandB} which states, informally, that all that is needed to obtain a certain degree of H\"older regularity in a random curve is a uniform estimate on the probabilities that the curve crosses an annulus $k$ distinct times.  In particular, we obtain the following result for $\SLE_\kappa$ in $\D$.  Let $A_r^R(z_0)$ denote the annulus with inner radius $r$ and outer radius $R$ centered at $z_0$.

\begin{theorem}\label{boundTheorem}
Fix $\kappa \le 4$. For any $k \ge 1$, there exists $c_k$ so that for any $z_0 \in \D$, $r < R$,
\[
\Prob\{\gamma \textrm{ traverses $A_r^R(z_0)$ at least $k$ separate times} \} \le c_{k} \; \Big(\frac{r}{R}\Big)^{\frac{\beta}{2}(\lfloor k/2 \rfloor - 1)}.
\]
\end{theorem}

Finally, there has been recent work defining a parametrization of $\SLE$, called the \emph{natural parametrization}, which should be the scaling limit for the parametrizations of the discrete curves (see \cite{nat1,nat2} for the theory in deterministic geometries, and \cite{nat3} for a version in random geometries).  It is conjectured that $\SLE_\kappa$ under the natural parametrization should have the optimal H\"older exponent our result indicates, however our techniques do not immediately illuminate this question.

The paper is organized as follows.  In Section~ \ref{prelims}, we review the basic definition of $\SLE_\kappa$ and introduce the notation used throughout.  Then, in Section~ \ref{regularity}, a brief overview of the regularity results needed from \cite{AandB} is given along with a discussion of their application to $\SLE_\kappa$. Before proving the regularity result, we present our applications by providing a definition of integration against an $\SLE$ path and the computation of the first three gradings of the expected signature in Section~ \ref{integrals}.  Finally, in Section~ \ref{crossings}, we prove the main estimate bounding the probability that an $\SLE_\kappa$ crosses an annulus at least $k$ times.

We consider $0 \le \kappa \le 4$ to be fixed and write $a\ceq 2/\kappa$.  All constants throughout may implicitly depend on $\kappa$.

\section{$\SLE$ definition and notation}\label{prelims}
We first review the definition of chordal $\SLE_\kappa$ in $\Half$.  For a complete introduction to the subject, see, for example, \cite{Lbook,parkcity,Werner}.  For any $\kappa \ge 0$, let $a = 2/\kappa$ and define $g_t(z)$ to be the unique solution to
\[
\bd_t g_t(z) = \frac{a}{g_t(z) + B_t}, \quad g_0(z) = z.
\]
where $B_t$ is a standard Brownian motion.  We refer to this equation as the \emph{chordal Loewner equation} in $\Half$, and the Brownian motion is the \emph{driving function}.

For any $z \in \Half$ this is well defined up to some random time $T_z$.  Let $H_t = \{z \in \Half \mid T_z > t\}$ be the set of points for which the solution is well defined up until time $t$.  The chordal Loewner equation is defined so that $g_t : H_t \rightarrow \Half$ is the unique conformal map from $H_t$ to $\Half$ which fixes infinity with $g_t(z) = z + \frac{at}{z} + O(z^{-2})$ as $z \rightarrow \infty$.

It was shown by Rohde and Schramm in \cite{RS} that for any value of $\kappa \neq 8$ there exists a unique continuous curve $\gamma : [0,\infty) \rightarrow \overline\Half$ such that $H_t = \Half \setminus \gamma[0,t]$.  This holds for $\kappa = 8$ as well, however the proof in this case differs significantly \cite{loop}. This curve is \emph{chordal $\SLE_\kappa$ from $0$ to $\infty$ in $\Half$}.  

To define $\SLE_\kappa$ in other simply connected domains $D$ from $z_1 \in \bd D$ to $z_2 \in \bd D$, let $f:\Half \rightarrow D$ be a conformal map so that $f(0) = z_1$ and $f(\infty) = z_2$, and define $\SLE_\kappa$ in this new domain by taking the image of the curve $\gamma$ under this conformal map.  

In this work, we mainly consider $\SLE_\kappa$ from $1$ to $-1$ in $\D$.  In this case, it is known that $\gamma(\infty^-) = -1$ (see, for example, \cite[Chapter 6]{Lbook} for the proof in $\Half$), and hence we extend $\gamma$ to be well defined on the times $[0,\infty]$.  While this particular choice of domain and boundary points is not required for the work that follows, it is important to choose a domain with sufficiently smooth boundary to avoid detrimental boundary effects (for instance, domains and boundary points without any H\"older continuous curves between them).

\section{Tortuosity, H\"older continuity, and dimension}\label{regularity}
\subsection{Definitions and relations}

To prove the claimed order of continuity, we need to use tools first described by Aizenman and Burchard in \cite{AandB}.  This review is devoid of proofs, which the interested reader may find in the original paper along with many results beyond what are needed in this paper.  To aid in this, we have included the original theorem numbers for each result with each statement.  We begin by describing three different measures of regularity and their deterministic relationships.  Throughout this section $\gamma:[0,1] \rightarrow \R^d$ is a compact continuous curve in $\R^d$.

First, recall that a curve $\gamma(t)$ is \emph{H\"older continuous of order $\alpha$} if there exists a constant $C_\alpha$ such that for all $0 \le s \le t \le 1$, we have that $|\gamma(s) - \gamma(t)| \le C_\alpha|t-s|^\alpha$.  This condition becomes stricter for larger values of $\alpha$, thus if one wants to turn this into a parametrization independent notion of regularity it makes sense to define 
\[
\alpha(\gamma) = \sup \{\alpha \mid \gamma \textrm{ admits an $\alpha$-H\"older continuous reparametization} \}.
\]

While a familiar and useful notion of regularity, it can be hard to work with directly for random curves.  A similar notion, which is more amenable to estimation is the concept of the \emph{tortuosity}.  Let $M(\gamma,\ell)$ denote the minimal number of segments needed to partition the curve $\gamma$ into segments of diameter no greater than $\ell$.  As with most of these dimension like quantities, we wish to understand its power law rate of growth, thus we define the \emph{tortuosity exponent} to be
\[
\tau(\gamma) = \inf \{ s > 0 \mid \ell^s M(\gamma,\ell) \rightarrow 0 \textrm{ as } \ell \rightarrow 0\}.
\]

These two notions are similar in so far as they define a type of regularity for a curve in a local way which is, to a large extent, insensitive to the large scale geometry of the curve.  As such, one should not be surprised that they are deterministically related by the following result.

\begin{theorem}[{\cite[Theorem 2.3]{AandB}}]
For any curve $\gamma:[0,1]\rightarrow \R^d$, 
\[
\tau(\gamma) = \alpha(\gamma)^{-1}.
\]
\end{theorem}

Often times, it is easier still to estimate a quantity which takes global geometry in to account, in particular we discuss the \emph{upper box dimension}.  Let $N(\gamma,\ell)$ denote the minimal number of sets of diameter at most $\ell$ needed to cover $\gamma$.  Then the upper box dimension is
\[
\dim_B(\gamma) = \inf \{ s > 0 \mid \ell^s N(\gamma,\ell) \rightarrow 0 \textrm{ as } \ell \rightarrow 0\}.
\]

The upper box dimension can differ quite markedly from the tortuosity exponent as a single set in the cover can contain a large number of different segments of $\gamma$ of similar diameter.  In fact, there exist curves in the plane which cannot be parametrized to be H\"older continuous of any order, and hence $\tau(\gamma) = \infty$, while $\dim_B(\gamma) \le d$ for any compact curve $\gamma:[0,1]\rightarrow \R^d$.  In general, it is immediate from the definitions that $\dim_B(\gamma) \le \tau(\gamma)$, however the inequality can be strict.

What is desired is a condition which deterministically ensures that the upper box dimension and the tortuosity exponent coincide, allowing us to control the optimal H\"older exponent with the upper box dimension.

Aizenman and Burchard provide such a property which they refer to as the \emph{tempered crossing property}.  Say that a curve $\gamma$ \emph{exhibits a $k$-fold crossing of power $\epsilon$ at the scale $r \le 1$} if it traverses some spherical shell of the form
\[D(x;r^{1+\epsilon},r) \ceq \{y \in \R^d \mid r^{1+\epsilon} \le |y-x| \le r\}.
\]  
With this, we say a curve has the tempered crossing property if for every $0 < \epsilon < 1$ there exists $k(\epsilon)$ and $0 < r_0(\epsilon) < 1$ such that on scales smaller than $r_0(\epsilon)$, the curve has no $k(\epsilon)$-fold crossings of power $\epsilon$.

\begin{theorem}[{\cite[Theorem 2.5]{AandB}}]
If $\gamma:[0,1]\rightarrow \R^d$ has the tempered crossing property,
\[
\tau(\gamma) = \dim_B(\gamma).
\]
\end{theorem}

Thus, the goal is to find a simple probabilistic condition which ensures that with probability one the tempered crossing property holds.  We present a weaker form of the theorem than is found in \cite{AandB}.

\begin{theorem}[{\cite[Lemma 3.1]{AandB}}]\label{ABTheorem}
	Let $\gamma:[0,1] \rightarrow \Lambda$ be a random curve contained in some compact set $\Lambda \subseteq \R^d$.  If for all $k$ there exists $c_k$ and $\lambda(k)$ so that for all $x \in \Lambda$ and all $0 < r \le R \le 1$ we have
	\[
	\Prob\{\gamma \textrm{ traverses $D(x;r,R)$ at least $k$ separate times} \} \le c_k \; \Big(\frac{r}{R}\Big)^{\lambda(k)}
	\]
	where additionally $\lambda(k) \rightarrow \infty$ as $k\rightarrow \infty$, then the tempered crossing probability holds almost surely, and hence
	\[
	\dim_B(\gamma) = \tau(\gamma) = \alpha(\gamma)^{-1}.
	\]
\end{theorem}

\subsection{$\SLE$ specific bounds}

We need two ingredients to apply the techniques of the previous section to $\SLE_\kappa$ and obtain Theorem~ \ref{regularityTheorem}.  

First, we need to prove that $\SLE_\kappa$ from $1$ to $-1$ in $\D$ satisfies the conditions of Theorem~ \ref{ABTheorem}.  This is the main work of this paper and the result, Theorem~ \ref{boundTheorem}, is proven in Section~ \ref{crossings}.  This estimate shows that for any $\kappa \le 4$ the tempered crossing property holds with probability one, and hence $\alpha(\gamma)^{-1} = \dim_B(\gamma)$.  Note that the condition that the curve has a finite parametrization is immaterial since we may turn the normally infinite parametrization of and $\SLE_\kappa$ curve to a finite one by precomposing by an appropriate function.

Second, we need to know the upper box dimension is $1+\kappa/8$ with probability one.  This is a consequence of a pair of well known results.  From \cite[Theorem 8.1]{RS} we obtain that the upper box dimension is bounded above by the desired value, while the lower bound can be obtained by noting the Hausdorff dimension is $1+\kappa/8$ almost surely (proven by Beffara in \cite{beffara}) and using that the Hausdorff dimension is a lower bound for the box dimension.

\section{Integrals and rough path theory}\label{integrals}
Before proving our regularity result, we discuss a few applications to integration along $\SLE$ paths and the rough path theory of $\SLE$.

\subsection{$d$-variation and integrals}

With the main regularity result, we may prove the existence of integrals of the form 
\[
\int_0^t f(s) \dd \gamma(s)
\]
when $\gamma$ is an $\SLE_\kappa$ and $f$ is a sufficiently nice function.  In particular, both the Young integral (as first defined in \cite{Young}, and used in rough path theory) and the integral defined by Lions in \cite{rough2} are well defined with probability one for $\SLE_\kappa$ with $\kappa \le 4$.  For simplicity we discuss only the Young integral -- checking the condition Lions' integral is similar.

Given a continuous curve $\gamma : [0,1] \rightarrow \R^d$, let $\|\gamma\|_p$ denote the $p$-variation of $\gamma$, defined as
\[
\|\gamma\|_p = \left[\sup_{\mathcal{P}} \sum_{i = 1}^{\#\mathcal{P}} |\gamma(t_i) - \gamma(t_{i-1})|^p\right]^{1/p}
\]
where the supremum is taken over partitions $\mathcal{P} = \{t_0, \ldots, t_n\}$ of $[0,1]$.  This notion of $p$-variation is not the one most commonly used elsewhere in probability which would have $\limsup_{|\mathcal{P}|\rightarrow 0}$ in place of the supremum where $|\mathcal{P}|$ is the mesh of the partition. Let $\mathcal{V}^p(\R^d)$ denote the set of all $\gamma:[0,1]\rightarrow\R^d$ with finite $p$-variation.

It is immediate from the definitions that if $\gamma : [0,1] \rightarrow \R^d$ is H\"older continuous of order $1/p$, then it is an element of $\mathcal{V}^p(\R^d)$.  Thus for $\SLE_\kappa$ in $\D$ from $1$ to $-1$, the main regularity result implies that a sample path $\gamma$ has finite $p$ variation for all $p > 1+\kappa/8$ with probability one.

The following theorem contains the definition of the Young integral (for a proof see, for example, \cite[Theorem 1.16]{rough1}).

\begin{theorem}\label{youngThm}
	Fix $p,q > 0$ such that $1/p + 1/q > 1$.  Take $f \in \mathcal{V}^q(\R)$, and $g \in \mathcal{V}^p(\R^d)$, then for every $t \in [0,1]$, we have
	\[
	\int_0^t f(s) \dd g(s) \ceq \lim_{|\mathcal{P}| \rightarrow 0} \sum_{j = 1}^{\#\mathcal{P}} f(t_j)(g(t_j) - g(t_{j-1})),
	\]
	where $\mathcal{P}$ is a partition of $[0,t]$, exists and is called the \emph{Young integral}.  Moreover, when considered as a function of $t$, the Young integral is an element of $\mathcal{V}^p(\R^d)$ and the integral depends continuously on $f$ and $g$ under their respective norms.
\end{theorem}

Thus, we may integrate any element of $\mathcal{V}^q(\R)$ for some $q < (8+\kappa)/\kappa$ against an $\SLE_\kappa$ sample path with probability one.  This includes functions such as Lipshitz functions of $\gamma$ itself (since $1+\kappa/8 < (8+\kappa)/\kappa < 2$ when $\kappa < 8$). Thus, via Picard iteration, we may define ordinary differential equations driven by $\SLE_\kappa$ in a path-wise manner (see \cite{rough1} for details of the general theory).

When working with these integrals, one often wants to apply results from standard calculus.  Luckily, this may frequently  be done using the following density result.  Given a curve $\gamma:[0,1] \rightarrow \R^d$ and a partition $\mathcal{P} = \{0=t_0,t_1, \ldots, t_n=1\}$ of $[0,1]$, let $\gamma^{\mathcal{P}}$ be the piecewise linear approximation to $\gamma$ obtained by linearly interpolating between $\gamma(t_i)$ and $\gamma(t_{i+1})$ for each $i$.

\begin{proposition}[{\cite[Proposition 1.14]{rough1}}]\label{approxPath}
Let $p$ and $q$ be such that $1 \le p < q$ and take $\gamma \in \mathcal{V}^p(\R^d)$.  Then, $\gamma^{\mathcal{P}}$ tends to $\gamma$ in $\mathcal{V}^q$ norm as the mesh of $\mathcal{P}$ tends to zero.  Additionally, this convergence may be taken simultaneously in supremum norm.
\end{proposition}

Using this approximation technique, statements about integrals against functions in $\mathcal{V}^p(\R^d)$ for $p < 2$ may be reduced to questions about the classical Stieltjes integral, as it and the Young integral are identical for piecewise linear functions.  In our case, we want to ensure that we may find an approximating sequence which is simple.  The proof of the following lemma is due to Laurence Field \cite{lf}.

\begin{lemma}\label{simpleLem}
Let $\gamma:[0,1] \rightarrow \C$ be a simple curve.  Then for any $\epsilon > 0$, there exists a partition $\mathcal{P}$ with the mesh of $\mathcal{P}$ less than epsilon and $\gamma^{\mathcal{P}}$ simple.
\end{lemma}
\begin{proof}
We prove a slightly stronger fact that there exists a partition $\mathcal{P}$ such that for each not only is the mesh of $\mathcal{P}$ smaller than $\epsilon$, but so is $|\gamma(t_{i})-\gamma(t_{i-1})|$ for all $i$.  Since $\gamma$ is simple, it is a homeomorphic to its image, and thus both $\gamma$ and $\gamma^{-1}$ are uniformly continuous.  Thus, we may find $\delta_1<\epsilon$, $\delta_2<\epsilon$, $\delta_3$, and $\delta_4$ so that
\begin{align*}
	|t-s| < \delta_1 & \implies |\gamma(t) - \gamma(s)| < \epsilon, \\
	|\gamma(t) - \gamma(s)| < \delta_2 & \implies |t-s| < \delta_1, \\
	|t-s| < \delta_3 & \implies |\gamma(t) - \gamma(s)| < \delta_2, \text{ and}\\
	|\gamma(t) - \gamma(s)| < \delta_4 & \implies |t-s| < \delta_3. \\
\end{align*}

Let $\Sigma$ be the set of all times $s \in [0,1]$ with times $0 = t_0 < t_1 < \ldots < t_n = s$ so that
\begin{itemize}
	\item $t_i - t_{i-1} < \epsilon$ and $|\gamma(t_{i})-\gamma(t_{i-1})| < \epsilon$ for all $i$,
	\item The curve $\eta$ obtained by concatenating the segments between $\gamma(t_{i-1})$ and $\gamma(t_i)$ is simple, and
	\item $\gamma^{-1}(\eta) \subseteq [0,s]$.
\end{itemize}

We first show that $\sup \Sigma \in \Sigma$.  Suppose not, and take times $0 = t_0 < t_1 < \cdots < t_n$ as above with $|\gamma(t_n) - \gamma(\sup \Sigma)| < \delta_4$.  By shortening the sequence of times $t_i$, we may assume that $\gamma(t_k)$ is closest to $\gamma(\sup \Sigma)$ when $k = n$.  By the choice of $\delta_4$, we know that $|t_n-\sup \Sigma| < \delta_3$.  Let $t_{n+1}$  be the maximum time $s$ so that $\gamma(s)$ is contained in the interval between $\gamma(t_n)$ and $\gamma(\sup \Sigma)$.  $t_{n+1}$ is at least $\sup \Sigma$ in size, so we will be done as long as $t_{n+1} \in \Sigma$.  The first condition is satisfied since since $|\gamma(t_{n+1})-\gamma(t_n)|<\delta_2<\epsilon$ and thus $t_{n+1}-t_n<\delta_1<\epsilon$.  The second condition holds since a non-trivial intersection of the final interval with any previous one would force either non-simplicity of $\gamma$ or a violation of the third condition for the curve up to time $t_n$.  The third condition holds by the definition of $t_{n+1}$.

We now show that $\sup \Sigma = 1$.  Suppose not, and take times $0 = t_0 < t_1 < \cdots < t_n = \sup \Sigma$ as above.  Take any $t \in (t_n, t_n+\delta_3)$ so that $\gamma(t_n)$ is closer to $\gamma(t)$ than it is to any of the intervals of $\eta$ not containing $\gamma(t)$.  As before, let $t_{n+1}$ be the maximum time $s$ such that $\gamma(s)$ is contained in the interval from $\gamma(t_n)$ and $\gamma(t)$ -- a time no smaller than $t$.  Analogously to before, one may readily check that this shows $t_{n+1} \in \Sigma$.  
\end{proof}

\subsection{Partial Expected Signature for $\SLE$}
Once iterated integrals are defined, we may understand the \emph{signature}, which is the fundamental object of study in rough path theory (See \cite{rough1} for a more detailed introduction to this field of study).  Of particular interest when dealing with random processes is the expected signature of the path \cite{expectedRough,expectedBrownianRough}.  

In this section we provide a computation of the first few gradings of the expected signature for $\gamma$, an $\SLE_\kappa$ from $0$ to $1$ in the disk of radius $1/2$ about $1/2$, which we denote by $D$.

Let $\gamma_1,\gamma_2 : [0,1] \rightarrow \R$ denote the real and imaginary components of $\gamma : [0,1] \rightarrow \C$ respectively and define the \emph{coordinate iterated integrals} as
\begin{align*}
\gamma^{k_1k_2\ldots k_n} & \ceq \limints_{0 <t_1 < t_2 < \ldots < t_n<1} \dd \gamma_{k_1}(t_1) \dd \gamma_{k_2}(t_2) \cdots \dd \gamma_{k_n}(t_n)\\
& = \int_0^1 \int_0^{t_n} \int_0^{t_{n-1}} \cdots \int_0^{t_2} \dd \gamma_{k_1}(t_1) \dd \gamma_{k_2}(t_2) \cdots \dd \gamma_{k_{n-1}}(t_{n-1})\dd \gamma_{k_n}(t_n)
\end{align*}
defining $\gamma^\emptyset \ceq 1$.  It is convenient to let $\mathbf{k} = k_1k_2 \ldots k_n$ denote the multi-index used above.  

An important computational tool when dealing with these iterated integrals is the notion of the \emph{shuffle product}, as defined in the following proposition.  We say a permutation $\sigma$ of $r+s$ elements is a $\emph{shuffle}$ of ${1, \ldots, r}$ and ${r+1, \ldots , r+s}$ if $\sigma(1) < \cdots < \sigma(r)$ and $\sigma(r+1) < \cdots < \sigma(r+s)$.

\begin{proposition}[{\cite[Theorem 2.15]{rough1}}]
Let $\gamma$ be in $\mathcal{V}^p(\R^d)$ for $p < 2$. Then
\[
\gamma^{k_1\ldots k_r} \cdot \gamma^{k_{r+1}\ldots k_{r+s}} = \sum_{\text{shuffles $\sigma$}} \gamma^{k_{\sigma^{-1}(1)}\ldots k_{\sigma^{-1}(r+s)}}.
\] 
\end{proposition}

We let $\mathbf{e}_{\mathbf{k}} \ceq \mathbf{e}_{k_1} \otimes \cdots \otimes \mathbf{e}_{k_n}$ denote the basis element for formal series of tensors on the standard basis of $\R^2$ (viewed as $\C$).  Then the \emph{signature} is defined to be
\[
\Sig(\gamma) = \sum_{\mathbf{k}}\gamma^{\mathbf{k}} \mathbf{e}_{\mathbf{k}}
\]
where the sum is taken over all multi-indices $\mathbf{k}$.  The expected signature is thus
\[
\E[\Sig(\gamma)] = \sum_{\mathbf{k}}\E[\gamma^{\mathbf{k}}] \mathbf{e}_{\mathbf{k}}.
\]

Computing the expected signature of any process is difficult (see, for example, \cite{expectedBrownianRough} for the computation of the expected signature of Brownian motion upon exiting a disk, where the solution is found in terms a recursive series of PDE), and thus computing the full expected signature of $\SLE_\kappa$ would be a major undertaking.  We provide a computation of the first three gradings.

To do so, we use the probability that a point in $D$ is above the curve $\gamma$.  This is a well known computation in the $\SLE$ literature, and was found by Schramm in \cite{percform} for $\SLE$ in the upper half plane from $0$ to $\infty$ in terms of a hyper-geometric function.  We use a simpler form which can be found in \cite{lawlerLeft}.  Let $\lambda = \beta - 1 = 4a-2 = 8/\kappa - 2$.  Then the probability that an $\SLE_\kappa$ from $0$ to $\infty$ in $\Half$ passes to the right of a point $r_0e^{i\theta_0}$ is
\[
\phi(\theta_0) \ceq C_\kappa \int_0^{\theta_0} \sin^\lambda(t) \dt \mwhere C_\kappa^{-1} \ceq \int_0^\pi \sin^\lambda(t) \dt = \frac{\sqrt{\pi}\Gamma(\frac{\lambda+1}{2})}{\Gamma(\frac{\lambda+2}{2})}.
\]

Given a point $x+iy \in D$, let $p(x,y)$ be the probability that an $\SLE_\kappa$ from $0$ to $1$ in $D$ passes below $x+iy$, which by conformal invariance can be obtained be pre-composing the above expression with the conformal map $z \mapsto iz/(1-z)$, which maps $D$ to $\Half$ fixing $0$ and sending $1$ to $\infty$.

\renewcommand{\arraystretch}{1.4} 
\begin{table}
	\begin{tabular}{|cc|c|}
		\hline
		$\lambda$ & $\kappa$ & $A_\kappa$ \\
		\hline
		$0$ & $4$ & $\frac{1}{48}$ \\
		$1$ & $\frac{8}{3}$ & $\frac{1}{48}(6\mathcal{K}-5)$ \\
		$2$ & $2$ & $\frac{1}{4}\log(2)-\frac{1}{6}$ \\
		$3$ & $\frac{8}{5}$ & $\frac{1}{96}(54\mathcal{K}-49)$ \\
		$4$ & $\frac{4}{3}$ & $\frac{2}{3}\log(2)-\frac{11}{24}$ \\
		$5$ & $\frac{8}{7}$ & $\frac{1}{128}(150\mathcal{K}-137)$ \\
		$6$ & $1$ & $\frac{6}{5}\log(2)-\frac{199}{240}$ \\
		\hline
	\end{tabular}
	\vspace{10pt}
	\caption{A collection of values of $A_\kappa$ across a range of integer values of $\lambda$.  $\mathcal{K} \ceq \sum_{i=1}^\infty \frac{(-1)^k}{(2k+1)^2} \approx 0.91596\ldots$ denotes Catalan's constant.}
	\label{AkTable}
\end{table}

\begin{proposition}\label{sigProp}
Fix $\kappa \le 4$ and let
\begin{align*}
A_\kappa & = \frac{1}{12} - \int_D yp(x,y) \dxy \\
& = \frac{C_\kappa}{4}\left[\int_0^{\pi/2} \frac{\sin(t)-t\cos(t)}{\sin^3(t)}\cos^\lambda(t)\dt\right] - \frac{1}{24}.
\end{align*}
Then 
\[
\E[\Sig(\gamma)] = 1 + \mathbf{e}_1 + \frac{1}{2}\mathbf{e}_{11} + \frac{1}{6}\mathbf{e}_{111} + A_\kappa \mathbf{e}_{122} - 2A_\kappa \mathbf{e}_{212} + A_\kappa \mathbf{e}_{221} + \cdots
\]
where $\gamma$ is an $\SLE_\kappa$ from $0$ to $1$ in $D$.
\end{proposition}

For integer values of $\lambda$ (which includes the values $\kappa = 4$, $8/3$, and $2$) this integral can also be evaluated nearly in closed form.  Several values of $A_\kappa$ for integer $\lambda$ may be found in Table~ \ref{AkTable}.  For an understanding of the qualitative behavior of $A_\kappa$ for other values of $\kappa$, we have included a graph in Figure~ \ref{AGraph}.  

\begin{figure}[!ht]
	\labellist
	\small
	\pinlabel $A_\kappa$ [b] at 18.5 136
	\pinlabel ${\scriptstyle 0.000}$ [r] at 15 5.5
	\pinlabel ${\scriptstyle 0.005}$ [r] at 15 36.5
	\pinlabel ${\scriptstyle 0.010}$ [r] at 15 67
	\pinlabel ${\scriptstyle 0.015}$ [r] at 15 97.5
	\pinlabel ${\scriptstyle 0.020}$ [r] at 15 128.5
	\pinlabel $\kappa$ [l] at 230 5
	\pinlabel ${\scriptstyle 0}$ [t] at 18.5 3
	\pinlabel ${\scriptstyle 1}$ [t] at 70.5 3
	\pinlabel ${\scriptstyle 2}$ [t] at 122 3
	\pinlabel ${\scriptstyle 3}$ [t] at 174 3
	\pinlabel ${\scriptstyle 4}$ [t] at 226 3
	\endlabellist
	\centering
	\includegraphics[scale=1.0]{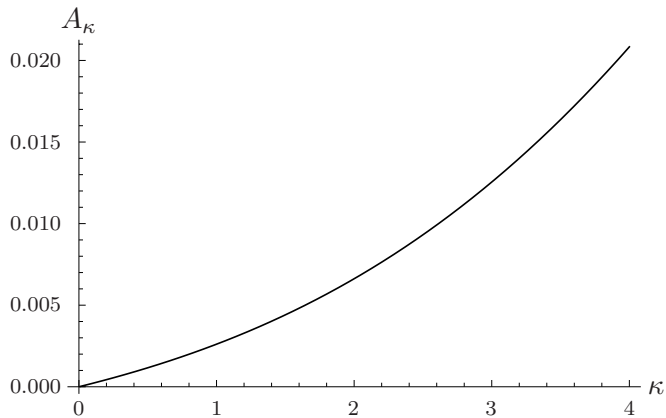}
	\caption{A graph of $A_\kappa$ as a function of $\kappa$.}
	\label{AGraph}
\end{figure}

\begin{proof}[Proof of Proposition~ \ref{sigProp}]
	First note the iterated integrals defining the signature exist since $\SLE_\kappa$ curves are in $\mathcal{V}^p(\C)$ for some $p < 2$.
	
	The initial $1$ occurs since $\gamma^\emptyset$ is defined to be $1$.  Since every $\SLE_\kappa$ is a curve from $0$ to $1$,
\[
\gamma^1 = \int_0^1 \dd \gamma_1(t) = \gamma_1(T) - \gamma_1(0) = 1, \qquad \gamma^2 = \int_0^1 \dd \gamma_2(t) = \gamma_2(T) - \gamma_2(0) = 0
\]
both with probability one, computing the first grading.  Thus, by considering these along with the shuffle products $\gamma^1 \gamma^1$, $\gamma^2 \gamma^2$, $\gamma^1\gamma^1 \gamma^1$, and $\gamma^2\gamma^2 \gamma^2$, we get the claimed values for $\gamma^{11}$, $\gamma^{22}$, $\gamma^{111}$, and $\gamma^{222}$ with probability one, and hence in expectation.

Next, the law of $\gamma$ is invariant under the map $\gamma \mapsto \bar \gamma$.  From the definition of the coordinate iterated integral, along with the definition of the Young integral one may see this implies 
\[
\E[\gamma^{k_1\ldots k_n}] = (-1)^{\#\{i \mid k_i = 2\}} \E[\bar \gamma^{k_1 \ldots k_n}]
\]
and hence any coordinate iterated integral with an odd number of imaginary components in its multi-index must have zero expectation as long as $\E[\gamma^{\mathbf{k}}]$ exists.  For the cases we need, one may repeat the Green's theorem argument which follows to show the there is some $C_{\mathbf{k}}$ so that $|\gamma^{\mathbf{k}}| \le C_{\mathbf{k}}$ with probability one and thus conclude immediately that $\E[\gamma^{\mathbf{k}}]$ exists.  Thus we have reduced the computation to the terms $\gamma^{122}$, $\gamma^{212}$, and $\gamma^{221}$.  By considering the shuffle products
\[
0 = \gamma^2\cdot\gamma^{12} = \gamma^{212} + 2\gamma^{122}, \qquad 0 = \gamma^2\cdot\gamma^{21} = 2\gamma^{221} + \gamma^{212}
\]
we need only compute $\E[\gamma^{221}]$.
	
	We use a version of Green's theorem for the Young integral.  Let $\eta$ be the concatenation of $\gamma$ with counter-clockwise arc from $1$ to $0$ along the boundary of the disk, and $A(\gamma)$ be the region enclosed within this simple loop.  In particular we wish to show that
	\[
	\int_0^1 \eta_2^2(t) \dd \eta_1(t) = -2\int_{A(\gamma)} y \dxy.
	\]
where the right integral should be understood in the Lebesgue sense.  

Given a curve $\gamma$ in $\mathcal{V}^p(\R^2)$ for some $p < 2$, by Proposition~ \ref{approxPath}, it is possible to approximate it arbitrarily well in $\mathcal{V}^q(\R^2)$ and supremum norm for some $p < q< 2$ by piecewise linear curves, which may be assumed to each be simple by Lemma~ \ref{simpleLem}.  Let $\gamma_n$ be such a sequence of approximations to the $\SLE$ curve $\gamma$.  Let $\eta_n$ be the concatenation of this piecewise linear approximation with the counter-clockwise arc from $1$ to $0$ along the boundary of the disk.  By the continuity properties of the Young integral, as stated in Theorem~ \ref{youngThm}, we know that
\[
\int_0 (\eta_n)_2^2(t) \dd (\eta_n)_1(t) \rightarrow \int_0^1 \eta_2^2(t) \dd \eta_1(t) \text{ as } n \rightarrow \infty.
\]
As these are now piecewise smooth, we know by Green's theorem that
\[
-2\int_{A(\gamma_n)} y \dxy \rightarrow \int_0^1 \eta_2^2(t) \dd \eta_1(t) \text{ as } n \rightarrow \infty.
\]
Thus, the argument is complete as long as
\[
\int_{A(\gamma_n)} y \dxy \rightarrow \int_{A(\gamma)} y \dxy \text{ as } n \rightarrow \infty.
\]
Since the $\gamma_n$ tend towards $\gamma$ in supremum norm, the desired convergence holds as long as the area within $\epsilon$ of the curve $\gamma$ tends to zero as $\epsilon$ tends to zero.  This is ensured by the almost sure order of H\"older continuity of $\gamma$, completing the proof of the required instance of Green's theorem.

One may see (by another similar approximation argument, or an application of the shuffle product $\gamma^2\gamma^2$ on the curve up to the time $t$) that
\begin{align*}
\gamma^{221} & = \int_0^1 \int_0^t \int_0^{t_2} \dd \gamma_2(t_1) \dd \gamma_2(t_2) \dd \gamma_1(t) \\
& = \frac{1}{2} \int_0^1 \gamma_2^2(t) \dd \gamma_1(t)
\end{align*}
and hence that
\[
\gamma^{221} = \frac{1}{2}\left(\frac{1}{8}\int_0^\pi \sin^3(\theta) \dd \theta - 2\int_{A(\gamma)} y \dxy\right) = \frac{1}{12} - \int_{A(\gamma)} y \dxy.
\]
Thus, all that needs to be understood is
\[
\E\left[\int_{A(\gamma)} y \dxy\right] = \int_Dy\Prob\{x+iy \in A(\gamma)\}\dxy.
\]
However, $\Prob\{x+iy \in A(\gamma)\} = p(x,y)$ and hence we have the first formula.

To obtain the more explicit formula, we need to work with the explicit definition of $p(x,y)$.  Let $g(z) = iz/(1-z)$ be the conformal map from $D$ to $\Half$ used in the definition of $p(x,y)$ and $f(w) \ceq w/(w+i)$ be its inverse.  Examining the integral, and changing variables to $\Half$ by setting $z = x+iy = f(w)$
\begin{align*}
	\int_Dyp(x,y)\dxy & = \int_D\Im(z)\phi(\arg(g(z)))\dd A(z) \\
	& = \int_\Half\Im(f(w))\phi(\arg( w))|f'(w)|^2\dd A(w) \\
	& = \int_\Half\Im(w/(w+i))\phi(\arg( w))|w+i|^{-4}\dd A(w)
\end{align*}
Changing to polar coordinates yields
\begin{align*}
	\lefteqn{\int_\Half\Im(f(w))\phi(\arg(w))|w+i|^{-4}\dd A(w)}\qquad \\
	& = -\int_0^\pi \phi(\theta)\cos(\theta)\int_0^\infty \frac{r^2}{(r^2+1+2r\sin\theta)^3}\dd r \dd \theta \\
	& = \int_0^{\pi/2} (1-2\phi(\theta))\cos(\theta)\int_0^\infty \frac{r^2}{(r^2+1+2r\sin\theta)^3}\dd r \dd \theta
\end{align*}
where the last line follows by the symmetries of $\sin$, $\cos$ and $\phi$.

A lengthy computation shows that the inner integral can be computed exactly.  The result is
\begin{align*}
	H(\theta) & \ceq \cos(\theta)\int_0^\infty \frac{r^2}{(r^2+1+2r\sin\theta)^3}\dd r \\
	& = \frac{(2\sin^2(\theta)+1)(\frac{\pi}{2} - \theta - \sin(\theta)\cos(\theta))}{8\cos^4(\theta)} - \frac{\tan(\theta)}{4}.
\end{align*} 
It will be convenient to later reparametrize, so note that
\[
H\Big(\frac{\pi}{2}-\theta\Big) = \frac{3\theta -2\theta\sin^2(\theta)-3\cos(\theta)\sin(\theta)}{8\sin^4(\theta)}.
\]
By inserting the definition of $\phi(\theta)$, reorganizing, and applying Fubini's theorem 
\begin{align*}
	\lefteqn{\int_0^{\pi/2} (1-2\phi(\theta))H(\theta) \dd \theta}\qquad \\
	& = 2C_\kappa \int_0^{\pi/2}  H(\theta) \int_\theta^{\pi/2} \sin^\lambda (t) \dt\dd \theta \\
	& = 2C_\kappa \int_0^{\pi/2}  H\Big(\frac{\pi}{2} - \theta\Big) \int_0^{\theta} \cos^\lambda (t) \dt\dd \theta \\
	& = 2C_\kappa \int_0^{\pi/2} \cos^\lambda (t) \int_t^{\pi/2}  H\Big(\frac{\pi}{2} - \theta\Big)\dd \theta   \dt.
\end{align*}
One may again compute the inner integral exactly and obtain
\[
\int_t^{\pi/2}  H\Big(\frac{\pi}{2} - \theta\Big)\dd \theta = \frac{1}{8}\left(1-\frac{\sin(t) - t\cos(t)}{\sin^3(t)}\right).
\]
Substituting this back in to the integral in question, and rearranging yields
\[
\int_Dyp(x,y)\dxy = \frac{1}{8}-\frac{C_\kappa}{4}\int_0^{\pi/2}\frac{\sin(t)-t\cos(t)}{\sin^3(t)}\cos^\lambda(t)\dt.
\]
\end{proof}

\section{Annulus crossing probabilities}\label{crossings}

In this section we prove our regularity result by providing the bound on annulus crossing probabilities for $\SLE$.

\subsection{Notation and Topology}\label{notation}

We let $B_r(z)$ denote the closed ball of radius $r$ around $z$, and $C_r(z)$ denote the circle of radius $r$ around $z$.

Let $A_r^R(z)$ denote the open annulus with inner radius $r$ and outer radius $R$ centered on $z$.  Let $\gamma:[0,\infty] \rightarrow \D$ be a chordal $\SLE$ from $1$ to $-1$ in the unit disk, considered under the standard capacity parametrization, and let $D_t$ be the component of $\D \setminus \gamma[0,t]$ which contains $-1$.

We wish to understand the probability that $\gamma$ crosses $A_r^R(z)$ $k$ times.  Fixing an annulus $A_r^R(z)$, let $\mathcal{C}_k =\mathcal{C}_k(z;r,R)$ denote the set of simple curves from $1$ to $-1$ that crosses the annulus precisely $k$ times.

To be precise in our definition of crossing, we define the following set of recursive stopping times.  In all the definitions, the infimums are understood to be infinity if taken over an empty set.  Let $\tau_0 = \inf \{t > 0 \mid \gamma(t) \not\in A_r^R(z)\}$.  This is the first time that the $\SLE$ is not contained within the annulus.  In the case that the annulus is bounded away from $1$, this time is zero.

We now proceed recursively as follows.  Assuming $\tau_i < \infty$, let $L_i$ be the random variable taking values in the set $\{I,O\}$ where 
\[
L_i = \begin{cases}
I & \gamma(\tau_i) \in B_r(z), \\
O & \gamma(\tau_i) \in \overline{\D \setminus B_R(z)}.
\end{cases}
\]
This random variable encodes the position of the curve at $\tau_i$, taking the value $I$ if it in inside the annulus, and $O$ if it is outside. 

Assuming $\tau_i < \infty$, define
\[
\tau_{i+1} = \begin{cases}
\inf\{ t > \tau_i \mid \gamma(t) \in B_r(z) \} & L_i = O, \\
\inf\{ t > \tau_i \mid \gamma(t) \in \overline{\D \setminus B_R(z)} \} & L_i = I.
\end{cases}
\]
In words, $\tau_{i+1}$ is the first time after $\tau_i$ that the curve $\gamma$ completes a traversal from the inside of the annulus to the outside, or from the outside of the annulus to the inside.  By continuity of $\gamma$, we know $\tau_{i+1} > \tau_i$.  We call the times $\tau_i$, for $i \ge 1$, \emph{crossing times}.  In this notation, $\mathcal{C}_k$ is precisely the set of curves such that $\tau_k < \infty$ and $\tau_{k+1} = \infty$.

Let $\sigma_i = \sup\{t < \tau_i \mid \gamma(t) \not \in A_r^R(z)\}$, which is to say the last entrance time of $\gamma$ in to the annulus, before crossing.  These are not stopping times, but it is useful to have them to aid in our definitions. We call the curve segments $\gamma_i \ceq \gamma[\sigma_i,\tau_i]$ \emph{crossing segments}. An illustration of the definitions so far are given in Figure~ \ref{initialDefs}.

\begin{figure}[!ht]
	\labellist
	\small
	\pinlabel $z$ [lt] at 50 50
	\pinlabel $\sigma_1$ [t] at 22 22
	\pinlabel $\tau_1$ [lt] at 48 30
	\pinlabel $\sigma_2$ [lt] at 32 60
	\pinlabel $\tau_2$ [rt] at 10 49
	\pinlabel $\sigma_3$ [l] at 74 82
	\pinlabel $\tau_3$ [lt] at 65 63
	\pinlabel $\sigma_4$ [lb] at 37 65
	\pinlabel $\tau_4$ [l] at 20 76
	\pinlabel $\sigma_5$ [l] at 43 89
	\pinlabel $\tau_5$ [t] at 45 69
	\pinlabel $\sigma_6$ [t] at 56 69
	\pinlabel $\tau_6$ [r] at 62 88
	\pinlabel $\sigma_7$ [r] at 14 69
	\pinlabel $\tau_7$ [l] at 34 62
	\pinlabel $\sigma_8$ [tr] at 70 48
	\pinlabel $\tau_8$ [tl] at 90 52
	\pinlabel $A_r^R(z)$ [c] at 70 30
	\pinlabel $\gamma$ [rb] at 20 90
	\endlabellist
	\centering
	\includegraphics[scale=2.5]{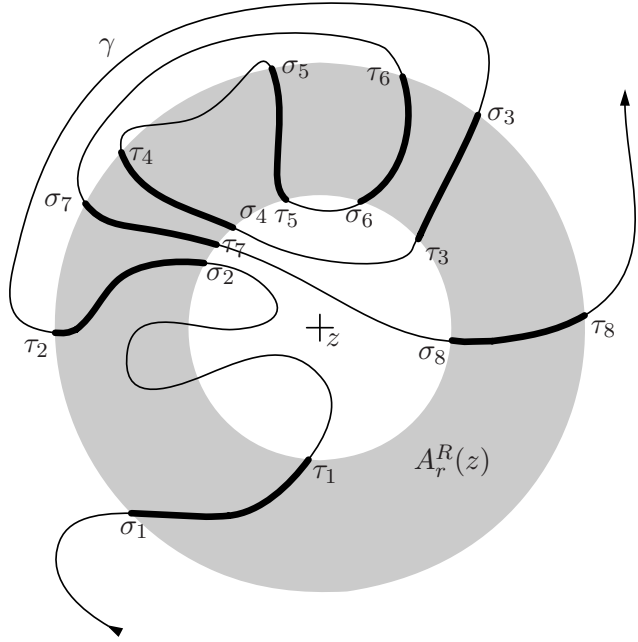}
	\caption{An illustration of the definitions given so far.  This picture should be understood as all strictly contained within $\D$.  All points along the curve $\gamma$ are labeled by the time the curve crosses the point, not by the point itself.  The crossing segments are indicated in bold.}
	\label{initialDefs}
\end{figure}

When we wish to estimate the probability that $\SLE$ performs various crossings, we will need some way of telling which crossings will require a decrease in probability.  For instance, in Figure~ \ref{initialDefs}, a crossing as between $\sigma_8$ and $\tau_8$ cannot be of small probability given the curve up to the time $\tau_7$ since the curve must leave the annulus in order to reach $-1$.  As we will see in Section~ \ref{crossingBounds}, the right way to handle this is to keep track of the \emph{crossing distance} of the tip to $-1$, which we denote by $\Delta_t$.

We define this notion as follows.  Let $E_t$, the set of extensions of $\gamma[0,t]$, denote the set of simple curves $\eta:[0,\infty] \rightarrow \D$ so that $\eta$ agrees with $\gamma$ up to time $t$ with $\eta(\infty) = -1$.  We then define
\[
\Delta_t = \min \{ k \ge 0 \mid \mathcal{C}_k \cap E_t \neq \emptyset\} - \#\{k > 0 \mid \tau_k \le t\},
\]
which is to say the minimum number of crossings needed after time $t$ to be consistent with the curve up to time $t$.  As an example, the sequence of values of $\Delta_{\tau_i}$ for $0 \le i \le 8$ from Figure~ \ref{initialDefs} are $(0,1,0,1,2,3,2,1,0)$.  By considering such examples, we 
quickly arrive at the following lemma.

\begin{lemma}\label{distLemma}
Given the above definitions, the followings statements all hold when $-1 \not \in A_r^R(z)$:
\begin{enumerate}
	\item $\Delta_t$ is integer valued and non-negative,
	\item given $t_1 < t_2$ so that there is no $i\ge 1$ with $t_1 < \tau_i \le t_2$, we have $\Delta_{t_1} = \Delta_{t_2}$, and
	\item $|\Delta_{\tau_{i+1}} - \Delta_{\tau_i}| = 1$.
\end{enumerate}
\end{lemma}
\begin{proof}
The first statement is immediate from the definition.  To prove the second, we proceed by showing that each side is an upper bound for the other.

First, $\#\{k > 0 \mid \tau_k \le t_1\} = \#\{k > 0 \mid \tau_k \le t_2\}$ since there is no $\tau_i$ with $t_1 < \tau_i \le t_2$.  Thus $\Delta_{t_2} \ge \Delta_{t_1}$ since every element of $E_{t_2}$ is an element of $E_{t_1}$.

To see the opposite inequality, we will take an element of $E_{t_1} \cap \mathcal{C}_{\Delta_{t_1}}$ and produce an element of $E_{t_2} \cap \mathcal{C}_{\Delta_{t_1}}$, thus proving the opposite inequality.  Let $\eta$ be such a curve in $E_{t_1} \cap \mathcal{C}_{\Delta_{t_1}}$.  Let $t^* = \max\{t \ge t_1 \mid \eta(t) \in \gamma[t_1,t_2]\}$, and let $t_* = \max\{t \le t_2 \mid \gamma(t) = \eta(t^*)\}$.  Both of these exist by compactness and continuity of $\gamma$ and $\eta$. We construct $\eta'$ as follows. First, follow the curve $\gamma$ up to time $t_2$, then follow the curve $\gamma$ backwards from $t_2$ until time $t_*$, and then follow $\eta$ from $t^*$ until it reaches $-1$.  

The curve $\eta'$ is not an element of $E_{t_2}$ since it retraces its path in reverse between $t_2$ and $t_*$, however otherwise it is simple.  By openness of $D_{t_2}$, we may perturb the curve so that after $t_2$, rather than retrace $\gamma$ exactly, it follows a similar path in $D_{t_2}$ which eventually continues as $\eta$, but still never crosses $A_r^R(z)$, yielding a curve $\eta''$ (see Figure~ \ref{topFig1} for an illustration of this process in an alternate case, which we will use later, where $t_1 = \tau_i$ and $t_2 = \tau_{i+1}$).  This curve does not have any more crossings than $\eta$ by construction, but also can have no fewer by our choice of $\eta$. thus $\eta''$ is the desired element in $E_{t_2} \cap \mathcal{C}_{\Delta_{t_1}}$.

We now prove the third item.  First note that $\Delta_{\tau_{i+1}} \ge \Delta_{\tau_i} - 1$ since $E_{\tau_{i+1}}$ is contained in $E_{\tau_{i}}$ and $\#\{k > 0 \mid \tau_k \le \tau_{i+1}\} = \#\{k > 0 \mid \tau_k \le \tau_{i}\}+1$.

By the same construction as above, we may take a curve $\eta \in E_{\tau_{i+1}}$ and produce a curve $\eta'' \in E_{\tau_{i}}$ with at most two more crossing of the annulus than $\eta$.  Thus $\Delta_{\tau_{i+1}} \le \Delta_{\tau_i} + 1$.

We can complete proof of the lemma as long as we can show that $\Delta_{\tau_{i+1}} \neq \Delta_{\tau_{i}}$.  However, this follows immediately since $-1 \not \in A_r^R(z)$ and hence the pairity of $\Delta_{\tau_{i}}$ must alternate (since we we know which boundary of $A_r^R(z)$ must be passed through last and $\gamma(\tau_i)$ alternates which boundary it is contained in by definition).
\end{proof}

\begin{figure}[!ht]
	\labellist
	\small
	\pinlabel $z$ [br] at 25 25
	\pinlabel $z$ [br] at 75 25
	\pinlabel $\tau_i$ [lt] at 46 26
	\pinlabel $\tau_i$ [lt] at 96 26
	\pinlabel $\tau_{i+1}$ [lt] at 30 30
	\pinlabel $\tau_{i+1}$ [lt] at 80 30
	\pinlabel $\gamma$ [r] at 26 10
	\pinlabel $\gamma$ [r] at 76 10
	\pinlabel $\eta$ [bl] at 45 40
	\pinlabel $\eta''$ [bl] at 95 40
	\endlabellist
	\centering
	\includegraphics[scale=2.5]{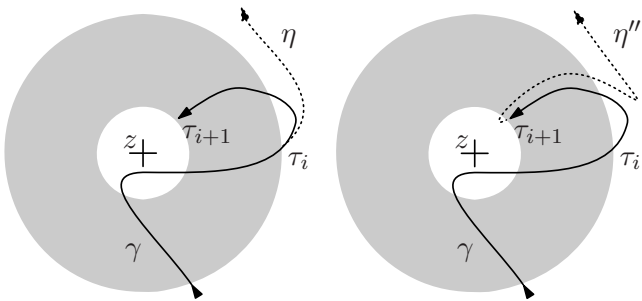}
	\caption{An example of the construction of $\eta''$ from $\eta$.}
	\label{topFig1}
\end{figure}

Note that everything besides the third bullet point in the above lemma would hold for any annulus, including those containing $-1$.  This issue will return later when our main estimate will need an extended proof when $-1 \in A_r^R(z)$.  

As $\Delta_t$ is constant on away from the crossing times, we often suppress the exact dependence on time, and let $\Delta_i \ceq \Delta_{\tau_i}$. As they will play a special role in the proof, we call the times $\tau_{i+1}$ such that $\Delta_{i+1} = \Delta_{i}+1$ the \emph{times of increase}, and the times such that $\Delta_{i+1} = \Delta_{i} - 1$ the \emph{time of decrease}.  We may further refine our understanding of the times of increase with the following lemma.

\begin{lemma}\label{increaseLemma}
Fix an annulus $A_r^R(z)$ not containing $-1$.  Let $\xi_1$ be the arc of $D_{\tau_{i}} \cap (C_r(z) \cup C_R(z))$ which contains $\gamma(\sigma_{i+1})$, and $\xi_2$ be the arc of $D_{\tau_{i}} \cap (C_r(z) \cup C_R(z))$ which contains $\gamma(\tau_{i+1})$.  Then, $\tau_{i+1}$ is a time of increase if and only if $\xi_1$ separates $\gamma(\tau_i)$ and $-1$ from $\xi_2$ in $D_{\tau_{i}}$.
\end{lemma}
\begin{proof}	
First, if $\xi_1$ separates $\xi_2$ from $-1$, every $\eta \in E_{\tau_{i+1}}$ intersects $\xi_1$ after $\tau_{i+1}$.  Let $t^* = \sup \{t \ge \tau_{i+1} \mid \gamma(t) \in \xi_1 \}$. Consider the curve $\eta' \in E_{\tau_{i}}$ which is constructed by following $\eta$ until $\sigma_{i+1}$, then following $\xi_1$ between the points $\gamma(\sigma_{i+1})$ and $\gamma(t^*)$ and then following $\eta$ again after $t^*$.  This curve has at least two fewer crossings of the annulus after the time $\tau_{i}$ than $\eta$ did, and hence by taking $\eta$ as a minimizer for the crossing distance, we see that $\Delta_{i+1} \ge \Delta_i +1$ and hence $\tau_{i+1}$ is a time of increase.

Thus we need only show the converse.  Note that $\xi_1$ always separates $\gamma(\tau_i)$ from $\xi_2$ in $D_{\tau_i}$ and thus we need only show $\xi_1$ separates $-1$ from $\xi_2$ in $D_{\tau_i}$.  We do so by showing that if $\xi_2$ is not separated from $-1$ in $D_{\tau_i}$ then $\tau_{i+1}$ must be a time of decrease.  To do so we preform a construction very similar to the previous lemma.  Take $\eta \in E_{\sigma_{i+1}}$ so that it minimizes $\Delta_{\sigma_{i+1}} = \Delta_i$.  First, note that $\eta$ may be assumed to be contained entirely in the component of $D_{\tau_i} \setminus \xi_1$ that contains both $\xi_1$ in its interior and and $-1$ in its boundary after the time $\sigma_{i+1}$ since otherwise we may follow very near to $\xi_1$ between $\sigma_{i+1}$ and $\eta$'s last crossing of $\xi_1$ and obtain a new curve that stays within the desired component and certainly has no more crossings than $\eta$.  

Now, we construct a curve $\eta' \in E_{\tau_{i+1}}$  with at most one more crossing of $A_r^R(z)$ than $\eta$.  We take $\eta'$ to be a simple curve very near to the curve formed by following $\gamma$ until time $\tau_{i+1}$ and then following the reversal of $\gamma$ back to time $\sigma_{i+1}$ and then following $\eta$ after time $\sigma_{i+1}$.  By our choice of $\eta$ to stay within $D_{\tau_{i}} \setminus \xi_1$ after $\sigma_{i+1}$, we may choose $\eta'$ to have at most one more crossing of $A_r^R(z)$ than $\eta$ (which is the crossing that occurred between $\sigma_{i+1}$ and $\tau_{i+1}$) and hence $\Delta_{i+1} \le \Delta_{i}$ showing $\tau_{i+1}$ is a time of decrease.
\end{proof}

\subsection{Crossing bounds}\label{crossingBounds}

With the above definitions, we may prove our main estimates.

First, we recall the definition of excursion measure, which is a conformally invariant notion of distance between boundary arcs in a simply connected domain (see, for example, \cite{Lbook}).  Let $D$ be a simply connected domain and let $V_1,V_2$ be two boundary arcs. As it is all we use, we assume the boundary arcs are $C^1$. If it were needed, conformal invariance would allows us to extend this definition to arbitrary boundaries. Let $h_D(z)$ denote the probability that a brownian motion started at $z$ exits $D$ through $V_2$.  Then the \emph{excursion measure} between $V_1$ and $V_2$ is defined to be
\[
\exc_D(V_1,V_2) = \int_{V_1} \bd_{\mathbf{n}} h_D(z) \adz
\]
where $\bd_{\mathbf{n}}$ denotes the normal derivative.  

Given a pair of disjoint simple $C^1$ curves $\xi_1,\xi_2 : (0,1) \rightarrow D$ in $D$, then we write $\exc_D(\xi_1,\xi_2)$ for the excursion measure between $\xi_1$ and $\xi_2$ in the unique component of $D\setminus (\xi_1(0,1), \xi_2 (0,1))$ which has both $\xi_1$ and $\xi_2$ on the boundary.

To relate this to probabilities involving $\SLE$, we need a lemma which can be found in \cite[Lemma 4.5]{twoPoint}.  The statement here is slightly modified from the version there, but the proof follows immediately from an application of the monotonicity of excursion measure.

\begin{lemma}\label{otherLem}
There exists a $c > 0$ so the following holds.  Let $D$ be a domain, and let $\gamma$ be a chordal $\SLE_\kappa$ path from $z_1$ to $z_2$ in $D$.  Let $\xi_1,\xi_2 : (0,1) \rightarrow $ be a pair of curves so that $\xi_i(0^+)$ and $\xi_i(1^-)$ are both in $\bd D$ so that $\xi_1$ separates $\xi_2$ from $z_1$ and $z_2$.  Then
\[
\Prob\{\gamma[0,\infty] \cap \xi_2(0,1) \neq \emptyset\} \le c \; \exc_D(\xi_1,\xi_2)^\beta
\]
where $\beta = 4a-1$.
\end{lemma}

We apply this lemma when $\xi_1$ and $\xi_2$ are arcs contained in $\bd A_r^R(z)$, and hence we wish to bound the size of the excursion measure between two such arcs.  We use the Beurling estimate (see, for example, \cite[Theorem 3.76]{Lbook}) as the main tool in providing this bound.  Given a Brownian motion $B_t$, let $\tau_{\D} = \inf\{ t > 0 \mid B_t \not \in \D\}$.

\begin{theorem}[Beurling estimate]\label{beurling}
	There is a constant $c < \infty$ such that if $\gamma : [0,1] \rightarrow \C$ is a curve with $\gamma(0) = 0$ and $|\gamma(1)| = 1$, $z \in \D$, and $B_t$ is a Brownian motion, then
	\[
	\Prob^z\{B[0,\tau_\D] \cap \gamma[0,1] = \emptyset\} \le c |z|^{1/2}.
	\]
\end{theorem}

Combining this with a number of standard Brownian motion estimates (see, for example \cite{bmbook}), we may provide the following estimate.

\begin{lemma}\label{exchEstimate}
There exists a $c< \infty$ so the following holds.  Let $r < R/16$, $\gamma:(0,1) \rightarrow A_r^R(z_0)$ be a curve with $\gamma(0^-) \in C_r(z_0)$ and $\gamma(1^+) \in C_R(z_0)$, and $U = A_r^R(z_0) \setminus \gamma(0,1)$. Let $\xi_1$ be an open arc in $C_r(z_0)$ subtending an angle $\theta_1$ and $\xi_2$ be an open arc in $C_R(z)$ subtending an angle $\theta_2$.  Then
\[
\exc_U(\xi_1,\xi_2) \le c \; \theta_1\theta_2\Big(\frac{r}{R}\Big)^{1/2}.
\]
\end{lemma}
\begin{proof}
We bound $h_U((r+\epsilon)e^{i \theta})$ by splitting into three steps: the probability the Brownian motion reaches radius $2r$ (providing the bound needed to take the derivative), the probability it reaches radius $R/2$ (providing the dependence on $r/R$), and finally the probability it hits $\xi_2$ if it reaches radius $R$ (providing dependence on $\theta_2$).  Integrating this bound over $\xi_1$ provides the desired bound on the excursion measure.

First, if the Brownian motion is to reach $\xi_2$, it must reach $C_{2r}(z)$.  By considering the gambler's ruin estimate applied to a Brownian motion motion started at $(r+\epsilon)e^{i\theta}$ in the annulus $A_r^{2r}(z_0)$, the probability that the Brownian motion reaches $C_{2r}(z_0)$ is at most
\[
\frac{\log (r+\epsilon) - \log (r)}{\log (2r) - \log (r)} = \log\Big(1+\frac{\epsilon}{r}\Big) \le \frac{\epsilon}{r}.
\]

Second, to estimate the probability that the Brownian motion travels from $C_{2r}(z_0)$ to $C_{R/2}(z_0)$ avoiding $\gamma$, we apply the Beurling estimate (Theorem~ \ref{beurling}).  This yields the bound of $c(r/R)^{1/2}$ by considering the curve in the annulus $A_{4r}^{R/2-2r}(\gamma(0^-))$ (which is non-degenerate, and has a ratio of radii comparable to $r/R$ since we assumed $r < R/16$).

Finally, we wish to estimate the probability that a Brownian motion starting on $C_{R/2}(z_0)$ hits an arc subtending an angle of $\theta_2$ located on $C_R(z_0)$.  By an explicit computation with the Poisson kernel in $\D$, we obtain that this probability is bounded above by $c\theta_2$.

Using the strong Markov property, we may combine these estimates yielding
\[
h_U((r+\epsilon)e^{i \theta}) \le c \; \epsilon \; \frac{1}{r} \theta_2 \Big(\frac{r}{R}\Big)^{1/2} , \quad
\bd_{\mathbf{n}}h_U(re^{i\theta}) \le c \; \frac{1}{r}\theta_2 \Big(\frac{r}{R}\Big)^{1/2}
\]
and hence
\[
\exc_U(\xi_1,\xi_2) = \int_{\xi_1} \bd_{\mathbf{n}} h_U(z) \adz \le c \; \theta_1\theta_2\Big(\frac{r}{R}\Big)^{1/2}.
\]
\end{proof}

The above lemma allows us to show the occurrence of a time of increase must be paid for with a corresponding cost in probability.
\begin{proposition}\label{increaseProp}
There exists a $c > 0$ so that for $i \ge 1$
\[
\Prob\{\tau_{i+1} < \infty \ ; \ \Delta_{i+1} = \Delta_{i} + 1 \mid \F_{\tau_i}\} \le c \;  \1\{\tau_i < \infty\}\Big(\frac{r}{R}\Big)^{\beta/2}.
\]
\end{proposition}
\begin{proof}
If $\tau_i = \infty$, the curve cannot cross again, and hence we may restrict to the complementary case.  

We wish to bound the probability that $\tau_{i+1} < \infty$ and it is a time of increase. Let $\xi_1$ and $\xi_2$ be a pair of arcs of $\bd A_r^R(z_0) \cap D_{\tau_i}$ which could contain $\sigma_{i+1}$ and $\tau_{i+1}$ respectively (by which we mean a pair of arcs with $\xi_1$ on the same component of $\bd A_r^R(z_0)$ as $\tau_i$ and $\xi_2$ on the opposite boundary component such that both arcs are in the boundary of a single component $U \in S_{\tau_i}$).  

Since $i \ge 1$ the annulus has been crossed at least once by time $\tau_i$.  Thus, $\xi_1$ and $\xi_2$ are contained in the boundary of some domain $U \in S_{\tau_i}$ which is bounded by some crossing segment $\gamma_j$.  By monotonicity of excursion measure, we may use Lemma~ \ref{exchEstimate} to conclude that
\[
\exc_{U}(\xi_1,\xi_2) \le \exc_{A_r^R(z_0) \setminus \gamma_j}(\xi_1,\xi_2) \le c \; \theta_1\theta_2 \Big(\frac{r}{R}\Big)^{1/2}
\]
where $\theta_1$ and $\theta_2$ are the angles subtended by the arcs $\xi_1$ and $\xi_2$.

Since $\tau_{i+1}$ is time of increase, Lemma~ \ref{increaseLemma} implies that $\xi_1$ separates $\xi_2$ from both $\gamma(\tau_i)$ and $-1$ as needed for Lemma~ \ref{otherLem}.  Restricting to such arcs we see 
\begin{align*}
\Prob\{\gamma[\tau_i,\infty] \cap \xi_2(0,1) \neq\emptyset \mid \F_{\tau_i} \} & \le c \; \exc_{D_{\tau_i}}(\xi_1,\xi_2)^\beta \\
& \le c \; \theta_1^\beta\theta_2^\beta \Big(\frac{r}{R}\Big)^{\beta/2} \\
& \le c \; \theta_1\theta_2 \Big(\frac{r}{R}\Big)^{\beta/2} \\
\end{align*}
where $c$ is being used generically, and the $\beta = 4a-1$ may be removed from the $\theta_i$ since $\theta_i \le 2\pi$ and $\beta \ge 1$ when $\kappa \le 4$.

We now conclude the bound by summing over all possible pairs of arcs satisfying the above criteria.  We use the extremely weak bound that perhaps \emph{every} pair of arcs, one on the interior boundary and one on the exterior boundary, might satisfy these conditions. By summing over all such pairs we see
\begin{align*}
\lefteqn{\Prob\{\tau_{i+1} < \infty \ ; \ \Delta_{i+1} = \Delta_{i} + 1 \mid \F_{\tau_i}\}}\quad \\
 & \le \sum_{\xi_1,\xi_2} \Prob\{\tau_{i+1} < \infty \ ; \ \Delta_{i+1} = \Delta_{i} + 1 \ ; \ \gamma(\sigma_{i+1}) \in \xi_1 \ ; \ \gamma(\tau_{i+1}) \in \xi_2 \mid \F_{\tau_i}\} \\
& \le \sum_{\xi_1,\xi_2} \Prob\{\gamma[\tau_i,\infty] \cap \xi_2(0,1) \neq\emptyset \mid \F_{\tau_i} \} \\
& \le c \sum_{\xi_1,\xi_2} \theta_1\theta_2 \Big(\frac{r}{R}\Big)^{\beta/2} \\
& = 4 \pi^2 c\; \Big(\frac{r}{R}\Big)^{\beta/2}.
\end{align*}
\end{proof}

The restriction that there is already at least one crossing is necessary in the above lemma.  By using the $\SLE$ Green's function (see, for example, \cite{nat1,twoPoint}), the probability of at least a single crossing is of the order $(r/R)^{2-d}$ for an annulus contained in $\D$ bounded away from $-1$ and $1$.  This is a weaker bound for $\kappa < 4$ than the one obtained above.  Additionally, an $\SLE$ must cross any annulus with $1$ and $-1$ in separate components of $\D\setminus A_r^R(z_0)$ at least once.   Since we need the bound to hold uniformly for all annuli but we do not need the exponents to be optimal, we use the trivial bound of $1$ for the first crossing.  More care must be taken in this estimate if a sharper exponent is desired.

We now obtain our bound on the number of annulus crossings by simple combinatorial estimates enforced by the form of $\Delta_i$ as a function of $i$ given in Lemma~ \ref{distLemma}.  These paths are closely related to \emph{Dyck~ paths}.  A Dyck path of length $k$ is a walk on $\N$ with $2k$ steps of $\pm 1$, which both starts and ends at zero.  Let $C_k$ denote the total number of Dyck paths of length $k$.  Since the path starts and ends at zero, there must be the same number of $+1$ steps as $-1$ steps.

\begin{theorem}\label{mainEstimate1}
There exist $c_1, c_2$ such that for any $k \ge 1$, $z_0 \in \D$, and $r < R$,
\[
\Prob\{\mathcal{C}_k(z_0;r,R)\} \le c_1 \; \Big(c_2 \frac{r}{R}\Big)^{\frac{\beta}{2}(\lfloor k/2 \rfloor - 1)}.
\]
\end{theorem}
\begin{proof}
Due to the topology of the situation, we must split this proof into two main cases: the case where $-1$ is not contained in $A_r^R(z_0)$ and the case where it is.
	
First, we prove the case $-1 \not\in A_r^r(z_0)$ as the second case reuses much of the same argument.  We proceed by splitting the event into those crossings which share a common sequence of values for $\Delta_i$, bounding the probability of a particular sequence by repeated application of Proposition~ \ref{increaseProp}, and then relating the number or such sequences to the number of Dyck Paths to obtain the constant.

Take some curve in $\mathcal{C}_k(z_0;r,R)$, and consider the associated $\Delta_i$ for $0 \le i \le k$.  In this case, $\Delta_k = 0$ since the curve proceeds to $-1$ without any further crossings.  Also, depending on if $\gamma(\tau_0)$ and $-1$ are in the same component of $\D \setminus A_r^R(z_0)$ or not, we have that $\Delta_0 \in \{0,1\}$ (this observation strongly uses that $\D$ is convex and hence can force at most one crossing of the annulus).  If $\Delta_0 = 1$,  $\Delta_1 = 0$ since $\gamma(\tau_1)$ would have to be the first time $\gamma$ was contained in the boundary of the component of $\D \setminus A_r^R(z_0)$ which contains $-1$. Thus by Lemma~ \ref{distLemma}, either $(\Delta_0, \ldots, \Delta_k)$ or $(\Delta_1, \ldots, \Delta_k)$ is a Dyck path of length $\lfloor k/2 \rfloor$ (and indeed since the number of steps in a Dyck path must be even, at most one of these cases can hold for any given $k$).  In either case the Dyck path must contain exactly $\lfloor k/2 \rfloor$ steps of $+1$, which is to say at least $\lfloor k/2 \rfloor$ times of increase.  

Thus, for a fixed Dyck path $\mathbf{d} = (d_0, \ldots, d_{2\lfloor k/2 \rfloor})$, we let $\mathcal{C}_{\mathbf{d}} \subseteq \mathcal{C}_k$ denote the set of curves for which either $(\Delta_0, \ldots, \Delta_k) = \mathbf{d}$ or $(\Delta_1, \ldots, \Delta_k) = \mathbf{d}$.  On this event, after discarding the first time of increase, and applying Proposition~ \ref{increaseProp} to each subsequent time of increase we obtain
\[
\Prob\{\mathcal{C}_{\mathbf{d}}\} \le c^{\lfloor k/2 \rfloor -1} \; \Big(\frac{r}{R}\Big)^{\frac{\beta}{2}(\lfloor k/2 \rfloor - 1)}.
\]

Finally, any such Dyck path $\mathbf{d}$ of length $\lfloor k/2 \rfloor$ can occur, and hence we need to sum over all of the possibilities yielding
\begin{align*}
	\Prob\{\mathcal{C}_k(z_0;r,R)\} &  = \sum_{\mathbf{d}} \Prob\{\mathcal{C}_{\mathbf{d}}\} \\
	& \le c^{\lfloor k/2 \rfloor -1} C_{\lfloor k/2 \rfloor} \; \Big(\frac{r}{R}\Big)^{\frac{\beta}{2}(\lfloor k/2 \rfloor - 1)}.
\end{align*}
Since $C_\ell \le 4^\ell$, there are universal $c_1$ and $c_2$ so that
\[
\Prob\{\mathcal{C}_k(z_0;r,R)\} \le c_1 \; \Big(c_2 \frac{r}{R}\Big)^{\frac{\beta}{2}(\lfloor k/2 \rfloor - 1)} 
\]
as needed for the first case.

In the second case, where $-1 \in A_r^R(z_0)$, we may no longer apply the above reasoning as Lemma~ \ref{distLemma} no longer holds for our annulus.  In particular, crossing the annulus need not change the number of crossings needed to reach $-1$ and hence we have no lower bound on the number of times of increase.  Let $r_0 = |z_0 + 1|$ be the distance between the center of the annulus and $-1$.  To extend the proof to this case, we split the annulus into the pair of annuli $A_I \ceq A_r^{r_0}(z_0)$ and $A_O \ceq A_{r_0}^R(z_0)$ and produce an upper bound of the exact same order by running the same argument in parallel for both annuli.

Take any curve that crosses $A_r^R(z_0)$ exactly $k$ times.  Any such curve must cross both $A_I$ and $A_O$ at least $k$ times.  Let $k_I$ be the number of times the curve crosses $A_I$ and similarly let $k_O$ be the number of times the curve crosses $A_O$.  Let $\Delta^I_i$ and $\Delta^O_i$ be the associated crossing functions.  As $-1$ is contained in neither $A_I$ nor $A_O$, Lemma~ \ref{distLemma} applies and, as above, we may associate a Dyck path of length $\lfloor k_I / 2 \rfloor$ to $\Delta^I_i$ and a Dyck path of length $\lfloor k_O / 2 \rfloor$ to $\Delta^O_i$.

We have no upper bound on $k_O$ and $k_I$ in comparison to $k$, and thus attempting to mirror the exact proof from above will not succeed since summing over all pairs of Dyck path of length at least $\lfloor k/2 \rfloor$ could yield a divergent sum since we have no control on the ratios $r/r_0$ and $r_0/R$.  Thus we must make this sum finite depending only on $k$.  Given a Dyck path of length $\ell' > \ell$ we must have at least $\ell$ times of increase in the first $2\ell$ steps of that Dyck path as otherwise the path would be negative at step $2\ell$.  Thus, as we only need $\lfloor k /2 \rfloor$ times of increase to obtain the desired bound, we need only consider the initial segments of the Dyck paths associated to $\Delta^I_i$ and $\Delta^O_i$.

Thus, for two initial segments $\mathbf{d_I}$ and $\mathbf{d_O}$ of Dyck paths both containing exactly $2\lfloor k/2 \rfloor$ steps, we let $\mathcal{C}_{\mathbf{d_I},\mathbf{d_O}} \subseteq \mathcal{C}_k$ denote the set of curves for which both either $(\Delta^I_0, \ldots, \Delta^I_k) = \mathbf{d}_I$ or $(\Delta^I_1, \ldots, \Delta^I_k) = \mathbf{d}_I$ and either $(\Delta^O_0, \ldots, \Delta^O_k) = \mathbf{d}_O$ or $(\Delta^O_1, \ldots, \Delta^O_k) = \mathbf{d}_O$ .  On this event, after discarding the first time of increase for both annuli, and applying Proposition~ \ref{increaseProp} to the next $\lfloor k/2 \rfloor - 1$ subsequent time of increase we obtain
\begin{align*}
\Prob\{\mathcal{C}_{\mathbf{d_I},\mathbf{d_O}}\} &  \le c^{2\lfloor k/2 \rfloor -2} \; \Big(\frac{r}{r_0}\Big)^{\frac{\beta}{2}(\lfloor k/2 \rfloor - 1)} \; \Big(\frac{r_0}{R}\Big)^{\frac{\beta}{2}(\lfloor k/2 \rfloor - 1)} \\
& = c^{2\lfloor k/2 \rfloor -2} \; \Big(\frac{r}{R}\Big)^{\frac{\beta}{2}(\lfloor k/2 \rfloor - 1)}.
\end{align*}

Finally, noting that there are at most $16^{\lfloor k/2 \rfloor}$ different pairs of $\mathbf{d_I},\mathbf{d_O}$, we may sum over each of these possibilities and obtain the desired bound in the same manner as the first case.
\end{proof}

To apply the results of Aizenman and Burchard, we need to have a bound on the probability of having at least $k_0$ crossings. Such bound is easily obtained from the above by summing over all $k \ge k_0$ to see that
\[
\Prob\{\gamma \textrm{ traverses $A_r^R(z_0)$ at least $k$ separate times} \} \le c_{k_0} \; \Big(\frac{r}{R}\Big)^{\frac{\beta}{2}(\lfloor k_0/2 \rfloor - 1)}
\]
where $c_{k_0}$ is some new constant depending only on $k_0$, thus completing the proof of the Theorem \ref{boundTheorem}, and hence of all results.
\section*{ Acknowledgments }
The author would like to thank Greg Lawler for helpful comments on an earlier draft of this paper, in particular for improving the explicit integral for $A_\kappa$ in Proposition~ \ref{sigProp} and Laurence Field for useful discussions.
\bibliography{AnnulusBib}{}

\begin{thebibliography}{10}

\bibitem{AandB}
M.~Aizenman and A.~Burchard.
\newblock H\"older regularity and dimension bounds for random curves.
\newblock {\em Duke Math. J.}, 99(3):419--453, 1999.

\bibitem{beffara}
V.~Beffara.
\newblock The dimension of the {SLE} curves.
\newblock {\em Ann. Probab.}, 36(4):1421--1452, 2008.

\bibitem{nat3}
B.~Duplantier and S.~Sheffield.
\newblock {S}chramm {L}oewner evolution and {L}iouville quantum gravity, 2010.
\newblock arXiv:1012.4800v1.

\bibitem{expectedRough}
T.~Fawcett.
\newblock {\em Problems in stochastic analysis: Connections between rough paths
  and non-commutative harmonic analysis.}
\newblock PhD thesis, Oxford Universy, UK, 2003.

\bibitem{lf}
L.~Field, 2011.
\newblock personal communication.

\bibitem{uniqueSig}
B.~Hambly and T.~Lyons.
\newblock Uniqueness for the signature of a path of bounded variation and the
  reduced path group.
\newblock {\em Ann. of Math. (2)}, 171(1):109--167, 2010.

\bibitem{holder2}
F.~Johansson~Viklund and G.~Lawler.
\newblock Optimal {H}older exponent for the {SLE} path, 2009.
\newblock arXiv:0904.1180v1, To appear in Duke Math. J.

\bibitem{Lbook}
G.~Lawler.
\newblock {\em Conformally invariant processes in the plane}, volume 114 of
  {\em Mathematical Surveys and Monographs}.
\newblock American Mathematical Society, Providence, RI, 2005.

\bibitem{parkcity}
G.~Lawler.
\newblock Schramm-{L}oewner evolution ({SLE}).
\newblock In {\em Statistical mechanics}, volume~16 of {\em IAS/Park City Math.
  Ser.}, pages 231--295. Amer. Math. Soc., Providence, RI, 2009.

\bibitem{loop}
G.~Lawler, O.~Schramm, and W.~Werner.
\newblock Conformal invariance of planar loop-erased random walks and uniform
  spanning trees.
\newblock {\em Ann. Probab.}, 32(1B):939--995, 2004.

\bibitem{nat1}
G.~Lawler and S.~Sheffield.
\newblock The natural parametrization for the {S}chramm-{L}oewner evolution,
  2009.
\newblock arXiv:0906.3804v1.

\bibitem{twoPoint}
G.~Lawler and B.~Werness.
\newblock Multi-point {G}reen's functions for {SLE} and an estimate of
  {B}effara, 2011.
\newblock arXiv:1011.3551v2.

\bibitem{nat2}
G.~Lawler and W.~Zhou.
\newblock {SLE} curves and natural parametrization, 2010.
\newblock arXiv:1006.4936.

\bibitem{lawlerLeft}
Gregory~F. Lawler.
\newblock Conformal invariance and 2{D} statistical physics.
\newblock {\em Bull. Amer. Math. Soc. (N.S.)}, 46(1):35--54, 2009.

\bibitem{holder1}
Joan~R. Lind.
\newblock H\"older regularity of the {SLE} trace.
\newblock {\em Trans. Amer. Math. Soc.}, 360(7):3557--3578, 2008.

\bibitem{rough2}
P-L. Lions.
\newblock Remarques sur l'int\'egration et les \'equations diff\'erentielles
  ordinaires, 2009.
\newblock Lecture available at:\\
  {\tiny\url{http://www.college-de-france.fr/default/EN/all/equ_der/Seminaire_du_6_novembre_2009_T.htm}}.

\bibitem{expectedBrownianRough}
T.~Lyons and H.~Ni.
\newblock Expected signature of two dimensional {B}rownian {M}otion up to the
  ﬁrst exit time of the domain, 2011.
\newblock arXiv:1101.5902v2.

\bibitem{rough1}
Terry~J. Lyons, Michael Caruana, and Thierry L{\'e}vy.
\newblock {\em Differential equations driven by rough paths}, volume 1908 of
  {\em Lecture Notes in Mathematics}.
\newblock Springer, Berlin, 2007.
\newblock Lectures from the 34th Summer School on Probability Theory held in
  Saint-Flour, July 6--24, 2004, With an introduction concerning the Summer
  School by Jean Picard.

\bibitem{bmbook}
P.~M{\"o}rters and Y.~Peres.
\newblock {\em Brownian motion}.
\newblock Cambridge Series in Statistical and Probabilistic Mathematics.
  Cambridge University Press, Cambridge, 2010.
\newblock With an appendix by Oded Schramm and Wendelin Werner.

\bibitem{RS}
S.~Rohde and O.~Schramm.
\newblock Basic properties of {SLE}.
\newblock {\em Ann. of Math. (2)}, 161(2):883--924, 2005.

\bibitem{first}
O.~Schramm.
\newblock Scaling limits of loop-erased random walks and uniform spanning
  trees.
\newblock {\em Israel J. Math.}, 118:221--288, 2000.

\bibitem{gff}
O.~Schramm and S.~Sheffield.
\newblock Contour lines of the two-dimensional discrete {G}aussian free field.
\newblock {\em Acta Math.}, 202(1):21--137, 2009.

\bibitem{percform}
Oded Schramm.
\newblock A percolation formula.
\newblock {\em Electron. Comm. Probab.}, 6:115--120 (electronic), 2001.

\bibitem{perc}
S.~Smirnov.
\newblock Critical percolation in the plane: conformal invariance, {C}ardy's
  formula, scaling limits.
\newblock {\em C. R. Acad. Sci. Paris S\'er. I Math.}, 333(3):239--244, 2001.

\bibitem{ising1}
S.~Smirnov.
\newblock Conformal invariance in random cluster models. {I}. {H}olomorphic
  fermions in the {I}sing model.
\newblock {\em Ann. of Math. (2)}, 172(2):1435--1467, 2010.

\bibitem{Werner}
W.~Werner.
\newblock Random planar curves and {S}chramm-{L}oewner evolutions.
\newblock In {\em Lectures on probability theory and statistics}, volume 1840
  of {\em Lecture Notes in Math.}, pages 107--195. Springer, Berlin, 2004.

\bibitem{Young}
L.~C. Young.
\newblock An inequality of the {H}\"older type, connected with {S}tieltjes
  integration.
\newblock {\em Acta Math.}, 67(1):251--282, 1936.

\end{thebibliography}
\bibliographystyle{plain}
\end{document}